\documentclass[11pt]{amsart}
\usepackage{tabularx,booktabs,tikz}
\usepackage{caption}
\usepackage{amsmath}
\usepackage{amsfonts}

\usepackage{amscd}
\usepackage{amsthm}
\usepackage{amssymb} \usepackage{latexsym}
\usepackage{eufrak}
\usepackage{euscript}
\usepackage{epsfig}
\usepackage{graphics}
\usepackage{array}
\usepackage{enumerate}
\usepackage{dsfont}
\usepackage{color}
\usepackage{wasysym}
\usepackage{hyperref}
\usepackage{pdfsync}

\newcommand{\bel}[1]{\begin{equation}\label{#1}}

\newcommand{\be}{\begin{equation}}

\newcommand{\ba}{\begin{eqnarray}}
\newcommand{\ea}{\end{eqnarray}}

\newcommand{\qe}{\end{equation}}
\newcommand{\R}{{\mathbb R}}
\newcommand{\N}{{\mathbb N}}
\newcommand{\Z}{{\mathbb Z}}

\newcommand{\HH}{\mathcal{H}}

\newcommand{\Hmm}[1]{\leavevmode{\marginpar{\tiny%
$\hbox to 0mm{\hspace*{-0.5mm}$\leftarrow$\hss}%
\vcenter{\vrule depth 0.1mm height 0.1mm width \the\marginparwidth}%
\hbox to
0mm{\hss$\rightarrow$\hspace*{-0.5mm}}$\\\relax\raggedright #1}}}

\newtheorem{theorem}{Theorem}[section]

\newtheorem{lemma}[theorem]{Lemma}
\newtheorem{corollary}[theorem]{Corollary}
\newtheorem{definition}[theorem]{Definition}

\newtheorem{prop}[theorem]{Proposition}

\newtheorem{claim}[theorem]{Claim}

\newcommand{\tm}{\begin{theorem}}
\newcommand{\tmd}{\end{theorem}}
\newcommand{\co}{\begin{corollary}}
\newcommand{\cod}{\end{corollary}}
\newcommand{\prp}{\begin{prop}}
\newcommand{\prpd}{\end{prop}}

\begin{document}

\title[Harmonic functions of polynomial growth on penny graphs]{Harmonic functions of polynomial growth on infinite penny graphs}

\author{Zunwu He}
\address{Zunwu He: School of Mathematical Sciences, Fudan University, Shanghai 200433, China}
\email{hzw@fudan.edu.cn}

\author{Bobo Hua}
\address{Bobo Hua: School of Mathematical Sciences, LMNS, Fudan University, Shanghai 200433, China; ; Shanghai Center for Mathematical Sciences, Jiangwan Campus, Fudan University, No. 2005 Songhu Road, Shanghai 200438, China.}
\email{bobohua@fudan.edu.cn}


\begin{abstract}


For an infinite penny graph, we study the finite-dimensional property for the space of harmonic functions, or ancient solutions of the heat equation, of polynomial growth.
We prove the asymptotically sharp dimensional estimate for the above spaces.

\end{abstract}
\maketitle



\par
\maketitle

\bigskip


\section{Introduction}
In geometric graph theory, penny graphs are contact graphs of unit circles in the plane. Finite penny graphs are extensively studied in the literature, e.g. \cite{Harborth74,Pollack85,Kupitz94,Pach95,Pach96,Csizmadia98,Pisanski00,Hlin01,Swanepoel09,Cerioli11,Eppstein18}. In this paper, we study discrete harmonic functions of polynomial growth on infinite penny graphs.

A penny graph is formed by arranging pennies in a non-overlapping way on the plane, making a vertex for each penny, and making an edge for each two pennies that touch.
Namely, let $\{C_i\}_{i=1}^\infty$ be a collection of circles of radius $\frac12,$ which are the boundaries of open disks $\{D_i\}_{i=1}^\infty$ in $\R^2$ (representing pennies),  such that $D_i\cap D_j=\emptyset,$ for any $i\neq j.$ Let $(V,E)$ be the contact graph of this configuration, i.e. $V=\{v_i\}_{i=1}^\infty$
where each $v_i$ represents the circle $C_i,$ and $\{v_i,v_j\}\in E$ if and only if $C_i$ and $C_j$ are tangent to each other. We call $(V,E)$ the \emph{penny graph of the configuration} $\{C_i\}_{i=1}^\infty.$ Note that it is a locally finite, simple, undirected graph. This graph has natural geometric realization in $\R^2,$ $$\phi:V\to \R^2,\quad \phi(v_i)=c_i,$$ where $c_i$ is the center of $C_i.$ For each edge $\{v_i,v_j\},$ we set $\phi(\{v_i,v_j\})$ to be the segment connecting $c_i$ and $c_j.$ This yields an embedding $\phi:(V,E)\to\R^2,$ and induces a CW complex structure, denote by $G=(V,E,F).$ Here $F$ is the set of faces, which corresponds to connected components of the complement of the embedding image of $(V,E).$  A \emph{penny graph} refers to the CW complex structure $G=(V,E,F)$ induced by a configuration of non-overlapping open disks of diameter $1.$  A penny graph $G=(V,E,F)$ is called connected if the 1-skeleton $(V,E)$ is connected. For any $\sigma\in F,$ we denote by $\deg(\sigma)$ the facial degree of $\sigma,$ and set $$D:=\sup_{\sigma\in F} \deg(\sigma).$$ We say that $G$ has (uniformly) \emph{bounded facial degree} if $D<\infty.$
In this paper, we only consider connected, infinite penny graphs with bounded facial degree, i.e. $D<\infty.$


For a graph $(V,E),$ the combinatorial Laplacian $\Delta$ is defined as, for any $f\in \R^V,$ the set of functions on $V,$
$$\Delta f(x):=\sum_{y} (f(y)-f(x)),\quad \forall x\in V,$$ where the summation is taken over neighbours of $x.$ A function $f$ is called harmonic if $\Delta f\equiv 0.$

We denote by $\Z^2$ the standard integer lattice graph in $\R^2.$ For any graph $(V,E),$ we denote by $d_{V}$ ($d$ in short) the combinatorial distance on the graph. The second author \cite{Hua20} proved that any infinite penny graph with bounded facial degree is quasi-isometric to $\Z^2,$ see Definition~\ref{def:quasi}, and hence the volume doubling property and the Poincar\'e inequality hold, see Definition~\ref{defi:vd}.

 For a graph $G=(V,E),$ fix $x_0\in V.$  For any $k\geq 0,$ we denote by $$\HH^k(G):=\left\{u: \Delta u\equiv 0, |u(x)|\leq C(1+d(x,x_0))^k, \forall x\in V\right\}$$ the space of harmonic functions of polynomial growth whose growth rate is at most $k.$
Note that the above space is independent of the choice of $x_0.$

On a Riemannian manifold with nonnegative Ricci curvature, the finite-dimensional property of the space of harmonic functions of polynomial growth was conjectured by Yau, and was confirmed by Colding-Minicozzi \cite{ColdingMinicozzi97}, see also \cite{ColdingMinicozzi98,ColdingMinicozzi98Weyl,Li97}. Following these arguments,  Delmotte \cite{Delm97} proved the finite-dimensional property for harmonic functions of polynomial growth on graphs under the assumptions of the volume doubling property and the Poincar\'e inequality. The following result was proved in \cite{Hua20}.

 \tm[Corollary~{3.10} in \cite{Hua20}]\label{thm:firstres} Let $G$ be an infinite penny graph with bounded facial degree. Then there exists $C(D)$ such that $$\dim \HH^k(G)\leq Ck^2,\quad \forall k\geq 1.$$
 \tmd

Note that for $\Z^2,$ $$\dim \HH^k(\Z^2)=2k+1,\quad k\in \N.$$ For a penny graph $G$ with bounded facial degree, it was conjectured in \cite[Conjecture~3.11]{Hua20} that $$\dim \HH^k(G)\leq Ck,\quad \forall k\geq 1.$$

For a planar graph with nonnegative combinatorial curvature, a rough dimensional estimate as in Theorem~\ref{thm:firstres} has been proved for the space of harmonic functions of polynomial growth, see \cite{HJLcrelle15}. Adopting an extension argument, the second author et al. \cite{HuaJosttrans15} obtained an asymptotically sharp dimensional estimate for that space. In this paper, we follow the proof strategy in \cite{HuaJosttrans15} to prove the above conjecture, which yields the asymptotically sharp estimate of $\dim \HH^k(G)$ for penny graphs.

\tm\label{thm:main1} Let $G$ be an infinite penny graph with bounded facial degree. Then there exists $C(D)$ such that $$\dim \HH^k(G)\leq Ck,\quad \forall k\geq 1.$$
\tmd

The proof strategy is as follows: The facial degree of each face of the penny graph $G$ is bounded above by $D.$ We first prove a crucial result, Theorem~\ref{lem:triangle}, that any face $\phi(\sigma)$ can be diagonally triangulated, see Definition~\ref{def:diagonal}, such that all angles in resulting triangles are uniformly bounded below by a constant depending on $D.$ Triangulating each face as above, we yield a triangulation $\mathcal{T}$ of $\R^2$ with each triangle in controllable size, i.e. bi-Lipschitz to a regular triangle of side-length one. We call $\mathcal{T}$ the \emph{associated triangulation} of $G.$ Given any $f\in \R^V,$ regarded as a function on the set of vertices of $\mathcal{T},$ we denote by $E(f)\in C(\R^2)$ the piecewise linear interpolation of $f$ to $\R^2.$ This yields an injective linear map: \begin{eqnarray*}&&E: \R^V\to C(\R^2), \\&&\quad \quad \ f\mapsto E(f).\end{eqnarray*} Note that the dimension of $\HH^k(G)$ equals the dimension of its image $E(\HH^k(G))$ in $C(\R^2).$ Instead of estimating the dimension of discrete harmonic functions on $G,$ we estimate that of extended functions on $\R^2.$ For any $f\in \HH^k(G),$ although $E(f)$ is not harmonic anymore, we prove that it satisfies the  mean value inequality in the large.
\tm\label{thm:mea}
Let $G=(V,E,F)$ be an infinite penny graph with bounded facial degree. Then there exist $R_1(D),C_1(D)$ such that for any $R\geq R_1,$ $p\in \R^2,$ any harmonic function $f$ on $G,$
\begin{equation}\label{eq:MVI}E(f)^2(p)\leq\frac{C_1}{r^2}\int_{D_{R}(p)}E(f)^2(y)dy,\end{equation} where $D_R(p)$ is the disk of radius $R$ centered at $p$ in $\R^2.$
\tmd
Using the mean value inequality and the homogeneous volume growth of $\R^2,$ we prove Theorem~\ref{thm:main1} by modifying the arguments in Colding-Minicozzi and Li \cite{ColdingMinicozzi97,ColdingMinicozzi98,ColdingMinicozzi98Weyl,Li97}. We remark that the extension method in this paper is much simpler than that in \cite{HuaJosttrans15}, which also applies for that case, i.e. for graphs with nonnegative combinatorial curvature.



The paper is organized as follows:
In next section, we recall some facts on penny graphs. In Section~\ref{sec:diag}, we prove the existence of associated triangulations of penny graphs, Lemma~\ref{lem:t1}. Section~\ref{sec:proof} is devoted to the proofs of Theorem~\ref{thm:mea} and Theorem~\ref{thm:main1}. In this paper, for the constants $C(D)$ depending only on $D,$ we simply write $C,$ and the constants $C$ may change from line to line.

\textbf{Acknowledgements.} B.H. is supported by NSFC, no.11831004 and no. 11926313.

\section{Preliminaries}
For a simple, undirected graph $(V,E),$ two vertices $x,y$ are called neighbors, denote by $x\sim y,$ if there is an edge connecting $x$ and $y.$ For any $x\in V,$ we denote by $\deg(x)$ the vertex degree of $x.$ The combinatorial distance $d$ on the graph is defined as, for any $x,y\in V$ and $x\neq y,$ $$d(x,y):=\inf\{n\in\N: \exists\{x_{i}\}_{i=1}^{n-1}\subset V, x\sim x_1\sim\cdots \sim x_{n-1}\sim y \}.$$  For any $x\in V$ and $R\geq 0,$ we denote by $$B_R(x):=\{y\in V: d(y,x)\leq R\}$$ the ball of radius $R$ centered at $x.$ We denote by $|\cdot |$ the counting measure on $V,$ i.e. for any $\Omega\subset V,$ $|\Omega|$ denotes the number of vertices in $\Omega.$
 The triple $(V,d,|\cdot|)$ is a discrete metric measure space. For any $\Omega\subset V,$ we denote by $$\delta\Omega:=\{y\in V\setminus \Omega: \exists x\in \Omega, y\sim x\}$$ the vertex boundary of $\Omega.$ We write $\overline{\Omega}:=\Omega\cup \delta\Omega.$ A function $f:\overline{\Omega}\to\R$ is called harmonic on $\Omega$ if $$\Delta f(x)=0,\ \forall x\in \Omega.$$

Let $G=(V,E,F)$ be a connected penny graph of a configuration of unit circles $\{C_i\}_{i=1}^\infty,$ with the embedding $\phi:(V,E)\to \R^2.$
One easily sees that
$$\deg(x)\leq 6,\quad \forall x\in V.$$
Note that for any vertices $x,y,$ $x\sim y$ if and only if $|\phi(x)-\phi(y)|=1,$ and $x\not\sim y$ if and only if $|\phi(x)-\phi(y)|>1.$ The embedding map $\phi$ induces the embedding of faces. We denote by $\phi(\sigma)$ the embedding image of $\sigma \in F$. Note that $\phi(\sigma)$ is a polygonal domain, which is homeomorphic to an open disk since $(V,E)$ is connected. The boundary of $\phi(\sigma)$ is a piecewise linear curve consisting of images of edges, which may have self-intersection. The facial degree of $\sigma,$ $\deg(\sigma)\in \N\cup\{\infty\},$ is the number of interior angles (or corners) of the polygonal domain $\phi(\sigma).$ 

From now on, we don't distinguish $G=(V,E,F)$ with its geometric realization $(\phi(V),\phi(E),\phi(F)),$
i.e. we identify $v\in V$ (resp. $e\in E,\sigma\in F$) with $\phi(v)$ (resp. $\phi(e),\phi(\sigma)$) in $\R^2.$

We recall the definition of the quasi-isometry between metric spaces.
\begin{definition}[\cite{BBI01}]\label{def:quasi} We say two metric spaces $(X,d_X)$ and $(Y,d_Y)$ are quasi-isometric if there exist a map $\psi:X\to Y,$ $a\geq 1,$ and $C>0$ such that
\begin{enumerate}[(a)]
\item for any $x_1,x_2\in X,$
$$a^{-1} d_X(x_1,x_2)-C \leq d_Y(\psi(x_1),\psi(x_2))\leq a d_X(x_1,x_2)+C,$$ and
\item for any $y\in Y,$ there exists $x\in X$ such that $d_Y(y,\psi(x))<C.$
\end{enumerate} 
\end{definition}

The following result was proved in \cite{Hua20}.
\begin{theorem}[Theorem~{1.1} and Theorem~{3.2} in \cite{Hua20}]\label{thm:quasi} Let $G=(V,E,F)$ be an infinite penny graph with bounded facial degree. Then for any $x,y\in V,$ $$\frac{d(x,y)}{2D}\leq |\phi(x)\phi(y)|\leq d(x,y).$$ Moreover, the metric space $(V,d_V)$ and $(\Z^2,d_{\Z^2})$ are quasi-isometric.
\end{theorem}

\begin{definition}\label{defi:vd}
For a graph $(V,E),$ we say that it satisfies
 the \emph{volume doubling property} if there exists $C_1$ such that \begin{equation}\label{eq:1}|B_{2R}(x)|\leq C_1|B_{R}(x)|, \quad \forall x\in V, R>0.\end{equation}
    We say that a graph $G=(V,E)$ satisfies the \emph{Poincar\'e inequality} if there exists a constant $C_2$ such that  for any $x\in V, R>0,$ and any function $f:B_{2R}(x)\to \R,$
  \begin{equation}\label{eq:poincare}\sum_{y\in B_R(x)}|f(y)-f_R|^2\leq C_2R^2\sum_{w,z\in B_{2R}(x)}|f(w)-f(z)|^2,\end{equation} where $f_R=\frac{1}{|B_R(x)|}\sum_{B_R(x)}f.$


\end{definition}
The next result was proven in \cite{Hua20}.
\tm[Theorem~{1.2} in \cite{Hua20}]Let $G=(V,E,F)$ be an infinite penny graph with bounded facial degree. Then the volume doubling property and the Poincar\'e inequality hold.
\tmd




\section{Diagonal triangulations of penny graphs}\label{sec:diag}
In this section, we prove the existence of proper triangulations of $\R^2$ associated to penny graphs.
\begin{definition}\label{def:diagonal}
By an \emph{$n$-polygon} $Q$ in the plane we mean a bounded closed subset in the plane such that the boundary $\partial Q$ consists of $n$ segments, there is a surjective immersion from the circle $S^1$ to $\partial Q$, and the preimage of any interior point of every segment is exactly one point. We call the segments the \emph{edges} of the polygon $Q.$ The endpoints of segments have $n$ preimages. We label the endpoints of segments according to their preimages, and call the labels the \emph{vertices} of $Q,$ which are not necessarily distinct. Hence $Q$ has $n$ edges and $n$ vertices, and the vertices can be arranged in clockwise way.

Let $V,E$ be the set of vertices and edges, respectively. For two vertices $v$, $w$, the segment $[v,w]\subset Q$ is not an edge of $Q$ with $[v,w]\cap V=\{v,w\}$, we call it a \emph{diagonal segment}. A triangulation of $n$-polygon is called a \emph{diagonal triangulation} if edges of the resulted triangles consist of the edges and diagonal segments of the polygon.

\end{definition}
Given $m$ points $x_1,x_2,\cdots,x_m\in \R^2$, we denote by
$$[x_1,x_2,\cdots,x_m]$$
the piecewise geodesic curve, i.e $[x_1,x_2,\cdots,x_m]=\bigcup\limits_{i=1}^{m-1}[x_i,x_{i+1}]$. We call $[x_1,x_2,\cdots,x_m]$ is simple curve if it has no self-intersection points.

We say $[x_1,x_m]$ is a \emph{$\epsilon$-follow} of $[x_1,x_2,\cdots,x_m]$, if there are

$$\hat x_1=x_1,\hat x_2,\cdots,\hat x_m=x_m$$
in the segment $[x_1,x_m]$ such that

$$d(x_1,\hat x_1)\leq \epsilon,d(x_2,\hat x_2)\leq \epsilon,\cdots,d(x_m,\hat x_m)\leq \epsilon.$$

For any $x,y$ in the plane, we denote by $l_{x,y}$ the line containing $x,y.$ $l_{x,y}$ has two ends, and one can equip it with an orientation from one end $-\infty$ to $x$, then to $y$ and finally to the other end $+\infty$. We always regard $l_{x,y}$ as such an oriented line.

Let $l_1,l_2$ be two lines in the plane. If the angle between $l_1$ and $l_2$ is less than or equal to $\theta\in [0,\pi/2]$, we say $l_1$ and $l_2$ are $\theta$-\emph{parallel}.

\begin{lemma}\label{diagonal triangulation}
Let $Q$ be an arbitrary $n$-polygon in the plane. Then there exists a diagonal triangulation of $Q$.
\end{lemma}

\begin{proof}
We argue in two possible cases.

Case1: The boundary of $n$-polygon is a simple closed curve. The existence of diagonal triangulation is well-known in the literature.

Case2: The boundary of $n$-polygon contains some loops.
We show the existence by an induction on $n$.

One easily sees that the sum of interior angles of $Q$ is $(n-2)\pi$ by Gauss-Bonnet formula, so one can choose a vertex $v$ at which the interior angle is less than $\pi$.

If this vertex does not coincide with other vertex, there are two adjacent vertices $w$, $u$. If $[w,u]$ is a diagonal segment of the polygon, this induces two sub-polygons with less vertices, and it follows from the induction.

If $[w,u]$ does not lie in the polygon, then there are some vertices contained in the interior of the triangle $\bigtriangleup wvu$. Otherwise, the interior angle at $v$ is at least $\pi$ which violates the choice of $v$. Then one can choose the nearest point $v'$ (with respect to $v$) among these vertices which are distinct to $v$ in $\bigtriangleup wvu$. Thus the segment $[v',v]\subset Q$ is a diagonal segment. It follows by induction.

If $v$ coincides with some other vertices, we can take two adjacent vertices $w$, $u$ such that there do not exist any other edges containing $v$ except edges $[w,v]$, $[u,v]$ in $\bigtriangleup wvu$. Then the proof is the same as above.

\end{proof}

\begin{theorem}\label{lem:triangle}
Let $Q$ be a $n$-polygon in a plane with $n$ vertices (not necessarily distinct), such that each edge is of length one and the distance between two distinct vertices is at least one. Then there exists a diagonal triangulation of $Q$, such that there is a positive lower bound depending only on $n$ for the interior angles of triangles.
\end{theorem}

\begin{proof}
Note that $n=3$ is trivial. We argue by contradiction for $n\geq 4$.  Assume it is not true, then there exists a sequence of $n$-polygons $P_k$ with some diagonal triangulations, such that there is a sequence of interior angles of triangles tending to zero. Since the boundary (relative to the plane) of such $n$-polygon is
a join of a simple closed curve with at most $n$ loops, the combinatorial structures and hence topological structures are finite. We may assume the boundaries of $\{P_i\}$, $\partial P_i$ are homeomorphic to each other and have the same combinatorial structure.
We proceed in the following steps.

\textbf{Step 1. The existence of the diagonal triangulation of such an $n$-polygon.}

This follows from Lemma \ref{diagonal triangulation}.

\textbf{Step 2. Construct the limit polygon.}


Let $v_1^k$, $v_2^k$, $\cdots $, $v_n^k$ be the vertices of the polygon $P_k$ arranged clockwise on the boundary, where $k=1$, $2$, $3$, $\cdots$. By translations we may assume $v^1_1=v^2_1=\cdots$ and
fix $v^k_1$, denoted by $o$. Therefore $P_k\subset B_n(o)$ for any $k$. We may assume $v^k_i\rightarrow v_i$ as $k\rightarrow \infty$ up to a subsequence, and $v^k_1=o=v_1$ for $1\leq i\leq n$.

Thanks to the existence of diagonal triangulation, we have finitely many diagonal triangulations for $P_k$. We say a diagonal triangulation $\Theta$ of $P_k$ is equivalent to a diagonal triangulation $\Lambda$ of $P_l$, if any resulted triangle of $\Theta$ contains vertices $v^k_i,v^k_j,v^k_s$ if and only if some resulted triangle of $\Lambda$ contains vertices $v^l_i,v^l_j,v^l_s$. This is indeed an equivalence relation.

One can denote by $\mathcal{T}_k=(V^k, E^k, F^k)$ for each polygon $P_k$, where $V^k$, $E^k$ and $F^k$ are the set of vertices, edges and diagonal segments, triangles of $P_k$, respectively. Therefore we may assume all polygons $P_k$ share one diagonal triangulation class $\Theta$ up to a subsequence.

Now one can construct a limit for these polygons $P_k$ as follows. For any triangle $\bigtriangleup v^k_iv^k_jv^k_m\in F^k$ of $P_k$ and any point $x^k\in \bigtriangleup v^k_iv^k_jv^k_m$, we have the unique coordinate expression in terms of $v^k_i,v^k_j,v^k_m$. This reads as
\begin{align*}
x^k=sv^k_i+tv^k_j+\mu v^k_m, \ s+t+\mu =1,\ s,\ t,\ \mu\in [0,1].
\end{align*}

Thus $x^k\rightarrow x\triangleq sv_i+tv_j+\mu v_m,k\rightarrow \infty$. By the above discussion, this gives a limit of polygons $P_k$, denoted by $P$. We also have

\begin{align*}
\lim\limits_{k\rightarrow \infty}\bigtriangleup v^k_iv^k_jv^k_m=\bigtriangleup v_iv_jv_m.
\end{align*}

If $\bigtriangleup v_iv_jv_m$ is non-degenerate, then it is obvious that the above limit gives a homeomorphism $\bigtriangleup v^k_iv^k_jv^k_m\approx \bigtriangleup v_iv_jv_m$ for any $k$. If $\bigtriangleup v_iv_jv_m$ is degenerate, we may assume $d(v_i,v_j)+d(v_j,v_m)=d(v_i,v_m)$. Hence a simple calculation shows that, $[v^k_i,v^k_m]$ is a diagonal segment and the degenerate triangle $\bigtriangleup v_iv_jv_m$ is obtained from $P_k$ by compressing the triangle $\bigtriangleup v^k_iv^k_jv^k_m$ to the diagonal segment $[v^k_i,v^k_m]$ in the topological view.

The polygon $P$ is topologically obtained from $P_k$ by compressing finitely many triangles in $F^k$ to diagonal segments. We describe a basic combinatorial and topological property in the following.
\begin{claim}
$ P=\lim\limits_{k\rightarrow \infty} P_k\approx P_k$. Moreover, $P$ is a $n$-polygon with $n$ vertices $v_1,v_2,\cdots,v_n$ and $n$ edges $[v_j,v_{j+1}]$ for $j=1,2,\cdots,n$ with setting $v_{n+1}=v_1$.
\end{claim}
\begin{proof}
By the above discussion, it suffices to show $\hat P_k\approx P_k$, $\hat P_k$ is a $n$-polygon with $n$ vertices $v_1,v_2,\cdots,v_n$ and $n$ edges $[v_j,v_{j+1}]$ for $j=1,2,\cdots,n-1$, where $\hat P_k$ is topologically obtained by compressing one triangle $\sigma\in F^k$ to a diagonal segment. Furthermore, it suffices to show this for the local neighbourhood $N(\sigma)$ of $\sigma$, where $N(\sigma)$ is the union of triangles having a diagonal segment of $\sigma$ in $F^k$. Denote by $\hat N(\sigma)$ the space obtained from $N(\sigma)$ by compressing $\sigma$ to a diagonal segment. There are three possible cases as follows.

\begin{enumerate}%
\item If $\sigma$ has three diagonal segments, then $N(\sigma)$ consists of four triangles in $F^k$. 
\item If $\sigma$ has one edge and two diagonal segments, then $N(\sigma)$ consists of three triangles in $F^k$. 
\item If $\sigma$ has two edges and one diagonal segment, then $N(\sigma)$ consists of two triangles in $F^k$.

\end{enumerate}
One can easily verify that the claim holds for $\hat N(\sigma)$ and $N(\sigma)$ in these cases.

Thus we conclude the above claim.
\end{proof}


\textbf{Step 3. Recover a uniform diagonal triangulation from the limit polygon.}

Take a diagonal triangulation $\Sigma$ of the limit polygon $P$, we will establish the following main lemma.
\begin{lemma}\label{main lemma}
For sufficiently large $k$, the segment $[v_i^k,v_j^k]\subset P_k$ whenever $[v_i,v_j]$ is a diagonal segement of the triangulation $\Sigma$ of $P$. Moreover, $[v_i^k,v_j^k]\rightarrow [v_i,v_j]$ as $k\rightarrow \infty$. In particular, $[v^k_i,v^k_j]$ is a diagonal segment of $P_k$.
\end{lemma}

This is the key point of the arguments.

By Step 2, the triangulations $\mathcal{T}_k=(V^k, E^k, F^k)$ give a generalized triangulation $\mathcal{T}=(V, E, F)$ of the limit polygon $P$. It is clear that some elements in $F$ are degenerate limit triangles. For $x\in V,E,F$, we denote by $x^k\in V^k,E^k,F^k$ respectively, such that $\lim\limits_{k\rightarrow \infty}x^k=x$. By the construction of the limit polygon, $x^k$ is uniquely determined by $x$ and thus is well defined.

For the diagonal segment $[v_i,v_j]\subset P$, we set
\begin{align}\label{vw}
[v_i,v_j]\cap E=\{v_i=w_1,w_2,\cdots,w_{K_1}=v_j\},
\end{align}
where $K_1$ does not exceed a number depending only on $n$ and $[v_i,v_j]=\bigcup\limits_{q=1}^{K_1-1}[w_q,w_{q+1}]$ with $[w_q,w_{q+1}]$ being contained in a regular limit triangle $\sigma_q\in F$. But some $w_q$ is in some degenerate limit triangles which are in the same line.

We call a point $x\in \{w_1,w_2,\cdots,w_{K_1}\}$ is a regular limit point, if it is contained in no degenerate limit triangles in $F$. Otherwise we call $x$ is a degenerate limit point.

For $w_1=v_i$ , set $z_1^k=v_i^k$. 

For $2\leq q\leq K_1-1$, then $w_q\notin V$ of $P$. Assume $w_q$ is a regular limit point in $F$ and $w_q\in [v_{q,1},v_{q,2}]$, where $[v_{q,1},v_{q,2}]\in E$. Let $z^k_q\in [v^k_{q,1},v^k_{q,2}]$ such that
$$d(z^k_q,v^k_{q,1})/d(v^k_{q,1},v^k_{q,2})=d(w_q,v_{q,1})/d(v_{q,1},v_{q,2}).$$

If $w_q$ is a degenerate limit point in $F$, $[w_{q-1},w_q]$ is contained in some regular limit triangle $\sigma_{q}\in F,$ then there exists $\sigma_{q}^k\in F^k$ satisfying $\lim\limits_{k\rightarrow \infty}\sigma_{q}^k=\sigma_{q}$.

Similarly, $[w_{q},w_{q+1}]$ is contained in some regular limit triangle $\sigma_{q+1}\in F$. Then there exists $\sigma_{q+1}^k\in F^k$ satisfying $\lim\limits_{k\rightarrow \infty}\sigma_{q+1}^k=\sigma_{q+1}$.
One may assume
$$w_q\in [v_{q,1},v_{q,2}]\subset \sigma_{q}$$
and
$$w_{q+1}\in [v_{q+1,1},v_{q+1,2}]\subset \sigma_{q+1}.$$
Let $z^k_{q,0}\in [v^k_{q,1},v^k_{q,2}]\subset \sigma_{q}^k$
satisfy
$$d(z^k_{q,0},v^k_{q,1})/d(v^k_{q,1},v^k_{q,2})=d(w_q,v_{q,1})/d(v_{q,1},v_{q,2}).$$

Let $\sigma^k_{q,1}\in F^k$ be a triangle containing $[v^k_{q,1},v^k_{q,2}]$ distinct to $\sigma^k_q$. One easily sees $\sigma_{q,1}\in F$ is a degenerate limit triangle. We may assume the vertices of $\sigma_{q,1}$ are $v_{q,1},v_{q,2},v_{q,3}$. Since $w_q\notin V$, it is in exactly one of the two shorter edges of $\sigma_{q,1}$. Assume $w_q\in [v_{q,1},v_{q,3}]$. Let $z^k_{q,1}\in [v^k_{q,1},v^k_{q,3}]\subset \sigma_{q,1}^k$ satisfy
$$d(z^k_{q,1},v^k_{q,1})/d(v^k_{q,1},v^k_{q,3})=d(w_q,v_{q,1})/d(v_{q,1},v_{q,3}),$$
and $\sigma_{q}^k\neq\sigma_{q,1}^k\in F^k$.

If $\sigma_{q,1}$ is a regular limit triangle in $F$, then one easily sees $\sigma_{q,1}=\sigma_{q+1}\in F$. Otherwise we can repeat the above procedures in finitely many steps to terminate with $\sigma_{q,a}=\sigma_{q+1}$ which is a regular limit triangle in $F$, where $a$ is a positive integer.

At the end, we get $z^k_{q,0},z^k_{q,1},z^k_{q,2},\cdots,z^k_{q,a+1}=z^k_{q+1}$ and the first $a+1$ terms share the same limit point $w_q$ as $k\to\infty$.

Thus we obtain a piecewise geodesic (segment) curve
$$[z^k_1,z^k_2,\cdots,z^k_K]=\bigcup\limits_{q=1}^{K-1}[z^k_q,z^k_{q+1}]\subset P_k$$
such that $[z^k_q,z^k_{q+1}]\subset \sigma_q\in F^k$ and $[z^k_1,z^k_2,\cdots,z^k_K]\rightarrow [v_i,v_j]$ (in the sense of the limit in Step 2 ) as $k\rightarrow \infty$ , where $z_1=w_1=v_i,z_K=w_{K_1}=v_j$ and $K$ does not exceed a number depending only on $n$.

Note that if we introduce an order $``\leq"$ to $[z_1,z_K]$ with $z_1\leq z_K$, then $z_1\leq z_2\leq \cdots\leq z_K$.

A useful observation is established as follows:
\begin{lemma}\label{disjoint}

For the piecewise geodesic curve
$$[z^k_1,z^k_2,\cdots,z^k_K],$$
$\sigma^k_1,\sigma^k_2,\cdots,\sigma^k_{K-1}$ are pairwise distinct triangles in $F^k$, where
$$[z^k_q,z^k_{q+1}]\subset \sigma^k_q.$$

\end{lemma}

\begin{proof}



Note that $\sigma^k_q,\sigma^k_{q+1}$ are distinct for $1\leq q\leq K-1$. Assume $\sigma^k_q=\sigma^k_p$ with $1\leq q\leq p-2<p\leq K$ and $\sigma^k_s\neq\sigma^k_r$ with any $q\leq s\leq r\leq p-1$.

Suppose that $\sigma_q$ is a degenerate limit triangle in $F$. We may assume $z_q\neq z_p$, then there exists some $e\in E$ which is also an edge of $P$ such that $$\emptyset\neq[z_1,z_K]\cap e\nsubseteq V.$$ This violates the definition of diagonal segment $[z_1=v_i,z_K=v_j]$. So that $z_q=z_p$, then $z_q=z_{q+1}=\cdots=z_p$. We can set $\hat z^k_{q+i}=z^k_{p-i}$ for $0\leq i\leq p-q$. This gives two ways to define $z^k_{p+1},z^k_{p+2},\cdots,z^k_q$, which is impossible by the construction of $[z^k_1,z^k_2,\cdots,z^k_K]$.

Suppose that  $\sigma_p$ is a regular limit triangle in $F$. On the other hand, $[z^k_q,z^k_{q+1},\cdots,z^k_p]\rightarrow [z_q,z_p]$ as $k\rightarrow \infty$, thus we must have $q=p-1$, which is a contradiction.

\end{proof}

We deduce the following corollary.
\begin{corollary}\label{simple segment}
The piecewise geodesic curve $[z^k_1,z^k_2,\cdots,z^k_K]$ is a simple curve in the plane, i.e $[z^k_1,z^k_2,\cdots,z^k_K]$ has no self-intersection points.
\end{corollary}


We may assume $d(z^k_s,z^k_t)$ is less than or equal to the circumference of $P$, i.e $d(z^k_s,z^k_t)\leq n$, where $k$ is sufficiently large and $1\leq s,t\leq K$.

Recall that for a degenerate limit point $w\in \{w_1,w_2,\cdots,w_{K_1}\}$, there exist $1\leq s_1\leq s_2\leq K$ such that $\{z^k_{s_1},z^k_{s_1+1},\cdots,z^k_{s_2}\}$ share the same limit point $w$ as $k\rightarrow \infty$. Moreover, $[z^k_s,z^k_{s+1}]\subset \sigma^k_s\in F^k$ with that $\sigma_s\in F$ is a degenerate limit triangle for $s_1\leq s\leq s_2-1$, and $[z^k_{s_1-1},z^k_{s_1}]\subset \sigma^k_{s_1-1}\in F^k$, $[z^k_{s_2},z^k_{s_2+1}]\subset \sigma^k_{s_2}\in F^k$ such that $\sigma_{s_1-1}$, $\sigma_{s_2}$ are regular limit triangles in $F$.

By the above discussion, for each $z^k_q$ with $2\leq q \leq K-1$, there exists the unique $e_q^k\in E^k$ containing $z_q^k$ and we denote by $l^k_q$ the line containing $e_q^k$. One assumes $z_q=w_{q\prime}$, and then there exists $e_{q\prime}\in E$ containing $w_{q\prime}$. Denote by ${l}_{q\prime}$ the line containing $e_{q\prime}$.

For $1\leq q_1, q_2,q_3\leq K$, we denote by $\alpha^k_{q_1,q_2,l_{q_2}}$, $\alpha^k_{q_1,q_2,l_{q_1}}$, $\alpha^k_{q_1,q_2,q_3}$ the angle between $[z^k_{q_1},z^k_{q_2}]$ and $l^k_{q_2}$, the angle between $[z^k_{q_1},z^k_{q_2}]$ and $l^k_{q_1}$, the angle between $[z^k_{q_1},z^k_{q_2}]$ and $[z^k_{q_2},z^k_{q_3}]$ respectively.

If $z_q=w_{q\prime}$ is a regular limit point in $F$. We denote by $\alpha_{q,1},\alpha_{q,2}\in (0,\pi)$ the angle between $[w_{q\prime-1},w_{q\prime}]$ and $\l_{q\prime}$, the angle between $[w_{q\prime},w_{q\prime+1}]$ and $\l_{q\prime}$, respectively. Then for sufficiently large $k$, we may assume
\begin{align}\label{regular angle}
|\alpha^k_{q-1,q,l_q}-\alpha_{q,1}|\leq 1/k,|\alpha^k_{q,q+1,l_q}-\alpha_{q,2}|\leq 1/k.
\end{align}
It is clear that $\alpha_{q,1}+\alpha_{q,2}=\pi$.

If $z_q=w_{q\prime}$ is a degenerate limit point in $F$, we may assume
 $$z^k_q\in  [z^k_{s_1},z^k_{s_1+1},\cdots,z^k_{s_2}]$$
as in the above with  $1\leq s_1\leq q\leq s_2\leq K$. We denote by $\alpha_{s_1,1},\alpha_{s_2,2}\in (0,\pi)$ the angle between $[w_{q\prime-1},w_{q\prime}]$ and $\l_{q\prime}$, the angle between $[w_{q\prime},w_{q\prime+1}]$ and $\l_{q\prime}$, respectively.
Then $z^k_q\in e^k_q=\sigma^k_{q-1}\cap\sigma^k_{q}$ and $\sigma^k_{q-1},\sigma^k_{q}$ are both degenerate limit triangles, for $s_1+1\leq q\leq s_2-1$. We may assume

\begin{align}\label{degenerate distance}
d(z^k_{q-1},z^k_q)\leq 1/k,
\end{align}

and

\begin{align}
|\alpha^k_{s_1-1,s_1,l_{s_1}}-\alpha_{s_1,1}|\leq 1/k,\ |\alpha^k_{s_2+1,s_2,l_{s_2}}-\alpha_{s_2,2}|\leq 1/k.\label{degenerate angle}
\end{align}
Since the distance between distinct vertices in $V$ is at least one, for sufficiently large $k$ and $s_1+1\leq k\leq s_2$, we can assume

\begin{align}\label{1/k-parallel}
l^k_{q-1},\ l^k_q\ \mathrm{are}\ 1/k-\mathrm{parallel}.
\end{align}

We interrupt with a useful and elementary lemma, and leave the proof to interested readers.
\begin{lemma}\label{parallel angle}
Let five points $x,o,o^\prime,y,y^\prime$ in the plane such that neither $x,o,y$ nor $x,o^\prime,y^\prime$ are in a line. If the line containing $o,y$ and the line containing $o^\prime,y^\prime$ are $\theta$-parallel, where $\theta\in [0,\pi/2]$. Then $|\angle xoy-\angle xo^\prime y^\prime|\leq \angle oxo^\prime+\theta$ if the product of two determinants $det(\overrightarrow{ox},\overrightarrow{oy})det(\overrightarrow{o^\prime x},\overrightarrow{o^\prime y^\prime})>0$.
\end{lemma}

We first demonstrate two special cases of the main lemma.

Case 1: $z^k_1,z^k_2,\cdots,z^k_{K-1}$ share the same limit point $z_1=z_2=\cdots=z_{K-1}=w_1=v_i$ as $k$ tends to $\infty$. We claim that there is a constant $c(k,K,n)$ such that $[z^k_1,z^k_K]$ is a $c(k,K,n)$-follow of $[z^k_1,z^k_2,\cdots,z^k_{K}]$ with $c(k,K,n)\rightarrow 0$ as $k\rightarrow \infty.$ Furthermore, $[z^k_z,z^k_K]\subset P_k$
for sufficiently large $k$ and $[z^k_1,z^k_K]\rightarrow [z_1=v_i,z_K=v_j]$ as $k\rightarrow \infty.$
\begin{proof}
Note that $d(z^k_1,z^k_K)=d(v^k_i,v_j^k)\geq 1$.
Since $z_1=z_2=\cdots=z_{K-1}=w_1=v_i$ is a degenerate limit point, by (\ref{degenerate distance}) we have $d(z^k_{q},z^k_{q+1})\leq 1/k$ for $1\leq q\leq K-2$. A simple calculation shows that, for simplicity, we can assume
\begin{align}\label{K-1,K,1}
\alpha^k_{1,K,K-1}\leq 1/k.
\end{align}
By (\ref{degenerate angle}), we have
\begin{align}\label{aK-1,2}
|\alpha^k_{K-1,K,l_{K-1}}-\alpha_{K-1,2}|\leq 1/k
\end{align}
and
\begin{align}\label{K-1,2}
\alpha_{K-1,2}\in (0,\pi).
\end{align}

We deduce that $l_{K-1}\cap l_{z^k_1,z^k_K}$ is exactly one point denoted by $\hat z^k_{K-1}$, for $k$ is sufficiently large.

If $\hat z^k_{K-1}\in (z^k_K,+\infty)$, then $\alpha^k_{1,K,K-1}=\alpha^k_{K-1,K,l_{K-1}}+\angle z^k_{K-1}\hat z^k_{K-1}z^k_K$. This is impossible by (\ref{K-1,K,1}), (\ref{aK-1,2}) and (\ref{K-1,2}), when $k$ is sufficiently large. Thus we may assume $\hat z^k_{K-1}\in (-\infty,z^k_K]$. By the Law of Sines, we get
$$d(z^k_{K-1},\hat z^k_{K-1})=d(z^k_{K-1},z^k_K)\sin\alpha^k_{1,K,K-1}/\sin\angle z^k_{K-1}\hat z^k_{K-1}z^k_K.$$

Using (\ref{K-1,K,1}), (\ref{aK-1,2}) and (\ref{K-1,2}) again, $d(z^k_{K-1},z^k_K)\leq n$ and $$\angle z^k_{K-1}\hat z^k_{K-1}z^k_K=\pi-\alpha^k_{K-1,K,1}-\angle z^k_Kz^k_{K-1}\hat z^k_{K-1}=\alpha^k_{K-1,K,1}-\alpha^k_{K-1,K,1},$$ we obtain
\begin{align}\label{K-1}
d(z^k_{K-1},\hat z^k_{K-1})\leq c(n,k), \mathrm{where}\,\, c(n,k)\rightarrow 0\,\, as\,\, k\rightarrow \infty.
\end{align}
If $\hat z^k_{K-1}\in (-\infty,z^k_1)$, by the fact that $z^k_{K-2},z^k_K$ are on different sides of $l_{K-1}$, we have
$$[z^k_1,z^k_2,\cdots,z^k_{K-2}]\cap [z^k_{K-1},\hat z^k_{K-1}]\neq \emptyset.$$

By (\ref{K-1}) we can assume $[z^k_{K-1},\hat z^k_{K-1}]$ is contained in some edge in $E^k$, which is impossible by Lemma \ref{disjoint} as $k$ is sufficiently large.
Thus we get $\hat z^k_{K-1}\in [z^k_1,z^k_K]$ with $d(z^k_{K-1},\hat z^k_{K-1})\leq c(k,n)$.

For any $2\leq q\leq K-2$, we need to show that there exists $\hat z^k_q$ satisfying $\{\hat z^k_q\}=l_q\cap [z^k_1,z_K^k]$ and $d(z^k_{q},\hat z^k_{q})\leq c(k,n)$.
Since $$d(z^k_q,z^k_1)\leq d(z^k_1,z^k_2)+\cdots+d(z^k_{q-1},z^k_{q})\leq q/k\leq K/k$$ and $d(z^k_1,z^k_K)\geq 1$, we have
\begin{align}\label{q,K,1}
\alpha^k_{q,K,1}\leq c(k,K)\rightarrow 0\,\,as\,\,k\rightarrow\infty.
\end{align}
Thus one can prove that
\begin{align}\label{q,K,K-1}
\alpha^k_{q,K,K-1}\leq \alpha^k_{q,K,1}+\alpha^k_{1,K,K-1}\leq c(k,K)\rightarrow 0\,\,as\,\,k\rightarrow\infty.
\end{align}

Note that $l_i,l_{i+1}$ are $1/k$-parallel for $1\leq i\leq K-2$, then $l_q,l_{K-1}$ are $K/k$-parallel. By (\ref{q,K,1}) (\ref{q,K,K-1}) and Lemma \ref{parallel angle}, one has
\begin{align}\label{K,q,lq}
|\alpha^k_{q,K,l_q}-\alpha^k_{K,K-1,l_{K-1}}|\leq c(k,K)\rightarrow 0\,\,as\,\,k\rightarrow\infty.
\end{align}

Recall that $d(z^k_q,z^k_K)\leq n$. Similar to the above argument, we can show that there exists $\hat z^k_q$ satisfying
$$\{\hat z^k_q\}=l_q\cap [z^k_1,z_K^k],$$

and

$$d(z^k_{q},\hat z^k_{q})\leq c(k,K,n)\rightarrow 0\,\,as\,\,k\rightarrow \infty.$$

Thus we obtain $[z^k_1,z^k_K]$ is a $c(k,K,n)$-follow of $[z^k_1,z^k_2,\cdots,z^k_{K}]$ with
$$c(k,K,n)\rightarrow 0,$$
as $k\rightarrow \infty.$ Observe that for sufficiently large $k$, $[z^k_q,\hat z^k_q]\subset \sigma^k_q\cap l_q$ with $2\leq q\leq K-1$, where $z^k_q\in e^k_q=\sigma^k_{q-1}\cap \sigma^k_q$. One can easily get $[\hat z^k_{q-1},\hat z^k_q]\subset \sigma^k_{q-1}\in F^k$ for $2\leq q\leq K,$ where we set $\hat z^k_{1}=z^k_1,\hat z^k_{K}=z^k_K.$

Finally, we obtain $[z^k_1,z^k_K]\subset\bigcup\limits^k_{q=2}[\hat z^k_{q-1},\hat z^k_q]\subset P_k$ for sufficiently large $k$ and $[z^k_1,z^k_K]\rightarrow [z_1=v_i,z_K=v_j]$ as $k\rightarrow \infty.$

\end{proof}

Case 2: No point in $\{z_1,z_2,\cdots,z_K\}$ is a degenerate limit point. We claim that there is a constant $c(k,K,n)$ such that $[z^k_1,z^k_K]$ is a $c(k,K,n)$-follow of $[z^k_1,z^k_2,\cdots,z^k_{K}]$ with $c(k,K,n)\rightarrow 0$ as $k\rightarrow \infty.$ Furthermore, $[z^k_z,z^k_K]\subset P_k$
for sufficiently large $k$ and $[z^k_1,z^k_K]\rightarrow [z_1=v_i,z_K=v_j]$ as $k\rightarrow \infty.$

\begin{proof}
For $2\leq q\leq K-1$, $[z^k_1,z^k_2,\cdots,z^k_q]$ is a simple curve by Corollary \ref{simple segment}. If $[z^k_1,z^k_2,\cdots,z^k_q]\cap(z^k_1,z^k_q)=\emptyset$, then there is a polygon in the plane such that its boundary is $[z^k_1,z^k_2,\cdots,z^k_q]\cup (z^k_1,z^k_q)$ by Jordan Curve Theorem. A simple computation shows that
\begin{align*}
\alpha^k_{2,1,q}+\alpha^k_{1,q,q-1}&=\pi-\alpha^k_{1,2,3}+\pi-\alpha^k_{2,3,4}+\cdots+\pi-\alpha^k_{q-2,q-1,q}\\&\leq (q-2)/k\leq K/k.
\end{align*}
Suppose that $[z^k_1,z^k_2,\cdots,z^k_q]\cap(z^k_1,z^k_q)\neq \emptyset$. In this subcase if $[z^k_1,z^k_q)\cap [z^k_{q-1},z^k_q]\neq \emptyset$, then $\alpha^k_{q-1,q,1}=0$. Otherwise, there is $x\in [z^k_1,z^k_2,\cdots,z^k_q]\cap(z^k_1,z^k_q)$ such that the subcurve of $[z^k_1,z^k_2,\cdots,z^k_q]$ from $x$ to $z^k_q$ does not intersect with $(z^k_1,z^k_q)$ and the argument similar to the above implies $\alpha^k_{q-1,q,1}\leq K/k$.
By symmetry, we have $\alpha^k_{2,1,q}\leq K/k$.

Thus for $2\leq q\leq K$, we have
\begin{align}\label{2,1,q}
\alpha^k_{q-1,q,1}\leq K/k,\,\,\alpha^k_{2,1,q}\leq K/k.
\end{align}
Similarly we have
\begin{align}\label{q+1,q,K}
\alpha^k_{q+1,q,K}\leq K/k,\,\,\alpha^k_{K-1,K,q}\leq K/k.
\end{align}

By (\ref{2,1,q}), we obtain for $2\leq q\leq K-1$
\begin{align}\label{q,1,K}
\alpha^k_{q,1,K}\leq \alpha^k_{2,1,q}+\alpha^k_{2,1,K}\leq 2K/k.
\end{align}

Similarly, we obtain for $2\leq q\leq K-1$
\begin{align}\label{q,K,1}
\alpha^k_{q,K,1}\leq 2K/k.
\end{align}

Applying $|\alpha^k_{q-1,q,l_q}-\alpha^k_{1,q,l_q}|\leq \alpha^k_{q-1,q,1}$, (\ref{2,1,q}) and $|\alpha^k_{q-1,q,l_q}-\alpha_{q,1}|\leq 1/k,$ we have
\begin{align}\label{1,q,lq}
|\alpha^k_{1,q,l_q}-\alpha_{q,1}|\leq (K+1)/k,\,\,\alpha_{q,1}\in (0,\pi).
\end{align}

Similarly, we have
\begin{align}\label{K,q,lq}
|\alpha^k_{K,q,l_q}-\alpha_{q,2}|\leq (K+1)/k,\,\,\alpha_{q,2}\in (0,\pi).
\end{align}

Recall that $\alpha_{q,1}+\alpha_{q,2}=\pi$, hence $z^k_1,z^k_K$ are on the different sides of $l_q$. This yields that there exists $\hat z^k_q$ such that $l_q\cap [z^k_1,z^k_K]=\{\hat z^k_q\}$. On the other hand, using the Law of Sines, (\ref{1,q,lq}), (\ref{q,1,K}) and $d(z^k_1,z^k_q)\leq n$, we have
\begin{align}\label{small dq}
d(z^k_q,\hat z^k_q)&=d(z^k_1,z^k_q)\sin\alpha^k_{q,1,K}/\sin\angle z^k_q\hat z^k_qz^k_1
\\&=d(z^k_1,z^k_q)\sin\alpha^k_{q,1,K}/\sin(\pi-\alpha^k_{q,1,K}-\alpha^k_{1,q,l_q})\nonumber
\\&\leq c(k,K,n)\rightarrow 0\,\,as\,\,k\rightarrow \infty.\nonumber
\end{align}

Thus, we show $[z^k_1,z^k_K]$ is a $c(k,K,n)$-follow of $[z^k_1,z^k_2,\cdots,z^k_K]$ for sufficiently large $k$. A same argument as in Case 1 yields that $[z^k_1,z^k_K]\subset P_k$ for sufficiently large $k$ and $[z^k_1,z^k_K]\rightarrow [z_1=v_j,z_K=v_j]$ as $k\rightarrow \infty.$
\end{proof}

Now we give the proof of Lemma~\ref{main lemma}.
\begin{proof}[Proof of Lemma~\ref{main lemma}]
We argue by induction on $s$, where $s$ is the number of degenerate limit points in $\{w_1,w_2,\cdots,w_{K_1}\}$ in (\ref{vw}).

For $s=0$, it is the Case 2.

For $s=1$, there are essentially two subcases.

Subcase 1: We may assume $z_1=z_2=\cdots=z_q$ with $z_{q+1}\neq z_q$ for $2\leq q\leq K-1$. Applying the argument in Case 2 to $[z^k_q,z^k_{q+1},\cdots,z^k_K]$, we have
\begin{align}\label{1q+1,q,K}
\alpha^k_{q+1,q,K}\leq K/k,
\end{align}
and $[z^k_1,z^k_2,\cdots,z^k_q,z^k_K]$ is also a simple curve in $P_k$. Then by (\ref{1q+1,q,K}) we get
\begin{align}\label{1K,q,lq}
|\alpha^k_{K,q,l_q}-\alpha_{q,2}|&\leq |\alpha^k_{K,q,l_q}-\alpha_{q+1,q,l_q}|+|\alpha_{q+1,q,l_q}-\alpha_{q,2}|
\\&\leq \alpha^k_{q+1,q,K}+1/k\nonumber
\\&\leq (K+1)/k.\nonumber
\end{align}

Next by applying Case 1 to $[z^k_1,z^k_{2},\cdots,z^k_q,z^k_K]$, for sufficiently large $k$ there is $\hat z^k_s$ such that $l_s\cap [z^k_1,z^k_K]=\{\hat z^k_s\}$ and
\begin{align}\label{q,K}
d(z^k_s,\hat z^k_s)\leq c(k,K,n)\rightarrow 0\,\,as\,k\rightarrow \infty,
\end{align}
where $s\in \{1,2,\cdots,q,K\}$.

Similarly, applying Case 1 to $[z^k_1,z^k_2,\cdots,z^k_{q+1}]$ we have
\begin{align}\label{1q+1}
&\qquad \quad|\alpha^k_{1,q+1,l_{q+1}}-\alpha_{q+1,1}|\leq (K+1)/k,
\\&|\alpha^k_{1,q+1,q+2}-\pi|\leq|\alpha^k_{1,q+1,q+2}-\alpha^k_{q,q+1,q+2}|+|\alpha^k_{q,q+1,q+2}-\pi|\label{lq+11}
\\&\qquad \qquad\qquad\ \ \leq \alpha^k_{1,q+1,q}+1/k\leq (K+1)/k.\nonumber
\end{align}

Thanks to (\ref{1q+1}) and (\ref{lq+11}), we can apply Case 2 to $[z^k_1,z^k_{q+1},z^k_{q+2},\cdots,z^k_K]$. Then for sufficiently large $k$ there is $\hat z^k_s$ such that $l_s\cap [z^k_1,z^k_K]=\{\hat z^k_s\}$ and
\begin{align}\label{q+1,K}
d(z^k_s,\hat z^k_s)\leq c(k,K,n)\rightarrow 0\,\,as\,k\rightarrow \infty.
\end{align}
where $s\in \{1,q+1,q+2,\cdots,K\}$.

Combining (\ref{q,K}) with (\ref{q+1,K}), we have $[z^k_1,z^k_K]$ is the $c(k,K,n)$-follow of $[z^k_1,z^k_2,\cdots,z^k_K]$ for lager enough $k$. And an argument similar to Case 1 yields that $[z^k_1,z^k_K]\subset P_k$ with $$[z^k_1,z^k_K]\rightarrow [z_1=v_i,z_K=v_j]\,\,as\,k\rightarrow \infty.$$

Subcase 2: We may assume $z_{s+1}=z_{s+2}=\cdots=z_t$ with $z_{s}\neq z_{s+1}$ and $z_{t+1}\neq z_t$ for $2\leq s+1<t\leq K-1$. Similar to the argument in Subcase 1, we apply Case 1 to $[z^k_s,z^k_{s+1},z^k_{q+2},\cdots,z^k_t]$, and then apply Case 2 to $[z^k_1,z^k_2,\cdots,z^k_{s},z^k_{t},z^k_{t+1},\cdots,z^k_K]$. This yields that for sufficiently large $k$ there is $\hat z^k_p$ such that $l_p\cap [z^k_1,z^k_K]=\{\hat z^k_p\}$ with
\begin{align}\label{s,t}
d(z^k_p,\hat z^k_p)\leq c(k,K,n)\rightarrow 0\,\,as\,k\rightarrow \infty,
\end{align}
where $p\in \{1,2,\cdots,s,{t},{t+1},\cdots,K\}$.

Note that for sufficiently large $k$ with $s+1\leq q\leq t$, we have $$d(z^k_q,z^k_K)\geq d(z_q,z_K)/2\geq d(w_{K_1-1},w_{K_1})/2>0$$ $$\mathrm{and}\ d(z^k_q,z^k_t)\leq (t-s)/k\leq K/k.$$ This implies $\alpha^k_{q,K,t}\leq c(k,K)\rightarrow 0$ as $k\rightarrow \infty$. Applying Case 2 to $[z^k_t,z^k_{t+1},\cdots,z^k_K]$, we can get
\begin{align}
\alpha^k_{t+1,t,K}\leq K/k.\label{t+1,t,K}
\end{align}

Recall that $l_q,l_t$ are $K/k$-parallel, by Lemma \ref{parallel angle} and (\ref{t+1,t,K}), we have
\begin{align*}
|\alpha^k_{K,q,l_q}-\alpha_{t,2}|
&\leq |\alpha^k_{K,q,l_q}-\alpha^k_{K,t,l_t}|+|\alpha^k_{K,t,l_t}-\alpha^k_{t+1,t,l_t}|+|\alpha^k_{t+1,t,l_t}-\alpha_{t,2}|
\\&\leq \alpha^k_{q,K,t}+K/k+\alpha^k_{t+1,t,l_t}+1/k
\\&\leq c(k,K)+(2K+1)/k,
\end{align*}
where $\alpha_{t,2}\in (0,\pi)$.

Then similar to the argument in Case 1, we have for sufficiently large $k$ there is $\hat z^k_q$ such that $l_q\cap (-\infty,z^k_K]=\{\hat z^k_q\}$ and
\begin{align}\label{11q}
d(z^k_q,\hat z^k_q)\leq c(k,K,n)\rightarrow 0\,\,as\,k\rightarrow \infty,
\end{align}
where $s+1\leq q\leq t$.

By symmetry, one can deduce that $|\alpha^k_{1,q,l_q}-\alpha_{s+1,1}|\leq c(k,K)+(2K+1)/k$, where $\alpha_{s+1,1}\in (0,\pi)$. This implies $\hat z^k_q\in [z^k_1,+\infty)$.

Thus by (\ref{11q}) and the above result, we have for sufficiently large $k$, there is $\hat z^k_q$ such that $l_q\cap [z^k_1,z^k_K]=\{\hat z^k_q\}$ and
\begin{align}\label{1q}
d(z^k_q,\hat z^k_q)\leq c(k,K,n)\rightarrow 0\,\,as\,k\rightarrow \infty,
\end{align}
where $s+1\leq q\leq t$.

Combining (\ref{s,t}) with (\ref{1q}), similar to Case 1, we finish the proof in Subcase 2.

Thus we finish the proof for $s=0,1$. For $s\geq 2$, we may assume $z_1,z_2,\cdots,z_p$ have exactly one degenerate limit point with $2\leq p\leq K-2$ and $z_p\neq z_{p+1}$. Using the case of $s=1$ to $[z^k_1,z^k_2,\cdots,z^k_{p+1}]$, we have for sufficiently large $k$, $[z^k_1,z^k_{p+1}]\subset P_k$ and the piecewise geodesic curve $$[z^k_1,z^k_2,\cdots,z^k_{p+1}]\rightarrow [z_1,z_{p+1}]\ as\ k\rightarrow \infty.$$

Note that $|\alpha^k_{1,p+1,l_{p+1}}-\alpha^k_{p,p+1,l_{p+1}}|\leq \alpha^k_{1,p+1,p}\leq c(k,K)\rightarrow 0$ as $k\rightarrow \infty.$

The piecewise geodesic curve $[z^k_1,z^k_{p+1},z^k_{p+2},\cdots,z^k_{K}]$ has $s-1$ degenerate limit points, hence one can use the induction assumption for
$$[z^k_1,z^k_{p+1},z^k_{p+2},\cdots,z^k_K],$$
and obtain for sufficiently large $k$ ,there is $\hat z^k_q$ such that $l_q\cap [z^k_1,z^k_K]=\{\hat z^k_q\}$ and
\begin{align}\label{2q}
d(z^k_q,\hat z^k_q)\leq c(k,K,n)\rightarrow 0\,\,as\,k\rightarrow \infty.
\end{align}
where $p+1\leq q\leq K$.

Similarly, using the induction assumption for
$$[z^k_{p+1},z^k_{p+2},\cdots,z^k_K]$$
and the case of $s=1$ to $[z^k_1,z^k_2,\cdots,z^k_{p+1},z^k_{K}]$, we have
for sufficiently large $k$ ,there is $\hat z^k_q$ such that $l_q\cap [z^k_1,z^k_K]=\{\hat z^k_q\}$ and
\begin{align}\label{3q}
d(z^k_q,\hat z^k_q)\leq c(k,K,n)\rightarrow 0\,\,as\,k\rightarrow \infty.
\end{align}
where $1\leq q\leq p+1$.

Combining (\ref{2q}) with (\ref{3q}), similar to Case 1, we complete the proof.

\end{proof}


Therefore, we show that for sufficiently large $k$, the segment $[v_k^i,v_k^j]$ is indeed a diagonal segment in $P_k$ whenever $[v_i,v_j]$ is a diagonal edge of the diagonal triangulation $\Sigma$ of $P$. Moreover $[v_i^k,v_j^k]\rightarrow [v_i,v_j]$ as $k\rightarrow \infty$.

Finally, the resulted triangles $\bigtriangleup v_iv_jv_m$ in $\Sigma$ of $P$ correspond to the triangles $\bigtriangleup v^k_iv^k_jv^k_m$ which gives a diagonal triangulation of $P_k$ for sufficiently large $k$. And the resulted triangles $\bigtriangleup v^k_iv^k_jv^k_m$ approximate those of $P$ in the sense of the limit in Step 2. This yields a contradiction.

\end{proof}

For each $\sigma\in F,$ there is a triangulation of $\sigma,$ denoted by $\mathcal{T}(\sigma),$ satisfying the properties in Theorem~\ref{lem:triangle}. We triangulate each $\sigma\in F$ as above, and obtain a geometric triangulation of $\R^2,$ denoted by $\mathcal{T}=(V',E',F'),$ where $V'=V,$ $E'$ (resp. $F'$) consists of edges (resp. faces) of $\mathcal{T}(\sigma),$ $\sigma\in F.$ The triangulation $\mathcal{T}$ is called an {associated triangulation} of $\R^2$ for the penny graph $G.$ We have the following lemma.

\begin{lemma}\label{lem:t1} Let $G$ be an infinite penny graph with bounded facial degree.
Then there exists an associated triangulation $\mathcal{T}=(V',E',F')$ such that
 for any triangle $\tau=\{v_1,v_2,v_3\}\in F',$ for any distinct $i,j,k\in \{1,2,3\},$
\begin{equation}\label{eq:dd1}\angle v_iv_jv_k\geq \frac{1}{C(D)},\quad 1\leq |v_i-v_j|\leq C(D),\end{equation}
where $\angle v_iv_jv_k$ is the angle at the vertex $v_j,$ and $C(D)$ is a constant depending on $D.$
\end{lemma}

\section{Proof of the main theorem}\label{sec:proof}
In this section, we prove the main results of the paper.

Let $G=(V,E,F)$ be an infinite penny graph with bounded facial degree, and $\mathcal{T}=(V',E',F')$ be an associated triangulation of $\R^2$ for $G,$ constructed in Lemma~\ref{lem:t1}. For any function $f\in \R^V,$
let $E(f):\R^2\to \R,$ denoted by $\overline{f}$ for short, be the piecewise linear interpolation of $f$ with respect to the triangulation $\mathcal{T},$ i.e. for any $\tau\in F',$
$E(f)\Huge|_{\overline{\tau}}$ is the linear interpolation of $f$ on the vertices of $\tau$ to $\overline{\tau}.$ This yields
\begin{eqnarray*}&&E: \R^V\to C(\R^2), \\&&\quad \quad \ f\mapsto E(f)=\overline{f}.\end{eqnarray*}

\begin{prop} Let $G=(V,E,F)$ be an infinite penny graph with bounded facial degree, and $\mathcal{T}$ be an associated triangulation of $\R^2.$
Then there exists $C(D)$ such that for any $f\in \R^V$ and $\tau\in F',$
\begin{equation}\label{eq:triest}\sum_{x\in \partial \tau\cap V} f^2(x) \leq C\int_\tau \overline{f}^2(y)dy.\end{equation}
 \end{prop}
 \begin{proof} Let $\partial \tau\cap V=\{v_1,v_2,v_3\}.$ By \eqref{eq:dd1}, there is a linear isomorphism $\eta$ of $\R^2$ such that
 $$\eta(v_1)=(0,0),\eta(v_2)=(1,0),\eta(v_3)=(\frac12,\frac{\sqrt{3}}{2}),$$ with $\frac{1}{C(D)} I\leq \eta^T\eta\leq C(D) I,$ where $I$ is the identity matrix.
 Then the assertion reduces to the one for a regular triangle with vertices $\{(0,0),(1,0),(\frac12,\frac{\sqrt{3}}{2})\}.$ The standard calculation yields the result.
 \end{proof}

By the volume doubling property and the Poincar\'e inequality, Delmotte \cite{Delmotte97} proved the mean value inequality for harmonic functions on graphs.
\begin{lemma}[\cite{Delmotte97}]\label{MVIOG}
Let $G=(V,E,F)$ be an infinite penny graph with bounded facial degree. Then there exists $C_1(D)$ such that for any $r>0,p\in V,$ any harmonic function $f$ on $B_{r}(p),$ we
have
\begin{equation}\label{MVI}f^2(p)\leq\frac{C_1}{|B_{r}(p)|}\sum_{x\in B_{r}(p)}f^2(x).\end{equation}
\end{lemma}

Now we can prove the mean value inequality for the extended function $\bar{f}$ on $\R^2$ for a harmonic function $f$ on $G.$
\begin{proof}[Proof of Theorem \ref{thm:mea}]
For any $p\in \R^2,$ there exists a triangular face $\tau_1\in F'$ such that $p\in\overline{\tau_1}.$ Then by the linear interpolation of $\bar f,$ there exists a vertex $q\in\partial \tau_1\cap V$ such that
\begin{eqnarray}\label{meanv1}\bar{f}^2(p)&\leq&f^2(q)\nonumber\\
&\leq&\frac{C}{|B_{r}(q)|}\sum_{y\in B_{r}(q)}f^2(y),
\end{eqnarray} where the last inequality follows from the mean value inequality (\ref{MVI}) for harmonic functions on the graph $G$.

Since $G$ is quasi-isometric to $\R^2,$ there exists $C$ such that
\begin{equation}\label{eq:quadratic}|B_r(x)|\geq Cr^2,\quad \forall x\in V, r>0.\end{equation}

Let $W_r:=\{\sigma\in F'\mid \overline{\sigma}\cap B_{r}(q)\neq\emptyset\}$ and $\overline{W_r}:=\bigcup_{\sigma\in W_r}\overline\sigma.$ For any $x\in \overline{W_r},$ there exists a face $\tau_2\in W_r$ such that $x\in\overline{\tau_2}.$ Set $z\in B_{r}(q)\cap \overline{\tau_2}.$ Then \begin{eqnarray*}|p-x|&\leq& |p-q|+|q-z|+|z-x|\leq C+d(q,z)+C\\
&\leq& r+2C\leq 2r,\quad \forall r\geq 2C.\end{eqnarray*}
This yields that for any $r\geq 2C,$ $$\overline{W_r}\subset  D_{2r}(p).$$
By (\ref{meanv1}) and \eqref{eq:quadratic}, we obtain for any $r\geq 2C,$
\begin{eqnarray*}\bar{f}^2(p)&\leq&\frac{C}{r^2}\sum_{y\in B_{r}(q)}f^2(y)\leq \frac{C}{r^2}\sum_{y\in \overline{W_r}\cap V}f^2(y)\\
&\leq&\frac{C}{r^2}\sum_{\sigma\in W_r}\sum_{y\in\partial \sigma\cap V}f^2(y)\\
&\leq&\frac{C}{r^2}\sum_{\sigma\in W_r}\int_{\sigma}\bar{f}^2\\
&\leq&\frac{C}{r^2}\int_{D_{2r}(p)}\bar{f}^2.
\end{eqnarray*} This proves the theorem.

\end{proof}

Since the extension map $E:\HH^k(G)\rightarrow C(\R^2)$ is injective, we estimate the dimension of $E(\HH^k(G)).$ Following the standard arguments by Colding-Minicozzi and Li, we will give the asymptotically sharp dimension estimate. The following lemmas follow verbatim as in \cite{Li97,Li12,ColdingMinicozzi97,ColdingMinicozzi98,Hua11,HuaJosttrans15}, hence we omit the proofs.

\begin{lemma}[Lemma~3.4 in \cite{Hua11}]\label{innerprod} For any finite dimensional subspace $K\subset E(\HH^k(G)),$ there exists a constant $R_0(K)$ depending on $K$ such that for any $R\geq R_0,$
\begin{equation}\label{innerp1}A_R(u,v)=\int_{D_R(p)}uv\end{equation} is an inner product on $K.$
\end{lemma}

\begin{lemma}[Lemma~{28.3} in \cite{Li12}]\label{PNGL1} Let $K$ be an $m$-dimensional subspace of $E(\HH^k(G)).$ Given $\beta>1,\delta>0,$ for any $R_1\geq R_0(K)$ there exists $R>R_1$ such that if $\{u_i\}_{i=1}^m$ is an orthonormal basis of $K$ with respect to the inner product $A_{\beta R},$ then $$\sum_{i=1}^{m} A_R(u_i,u_i)\geq m\beta^{-(2d+2+\delta)}.$$
\end{lemma}

By the mean value inequality for extended functions, Theorem~\ref{thm:mea}, we prove the following lemma.
\begin{lemma}[Lemma~{28.4} in \cite{Li12}, Lemma~{28.4} in \cite{HuaJosttrans15}]\label{PNGL2} Let $G=(V,E,F)$ be an infinite penny graph with bounded facial degree, $K$ be an $m$-dimensional subspace of $E(\HH^k(G)).$ Let $\epsilon\in(0,\frac{1}{2}).$ Then there exists constants $R_2(D,\epsilon)$ and
$C(D)$ such that for any basis of $K,$ $\{u_i\}_{i=1}^m,$ and $R\geq R_2,$ we have
$$\sum_{i=1}^mA_R(u_i,u_i)\leq C\epsilon^{-1}\sup_{u\in <A,U>}\int_{D_{(1+\epsilon)R}(p)}u^2,$$ where $<A,U>:=\{w=\sum_{i=1}^ma_iu_i:\sum_{i=1}^ma_i^2=1\}.$
\end{lemma}

Now we are ready to prove Theorem~\ref{thm:main1}.
\begin{proof}[Proof of Theorem~\ref{thm:main1}] For any $m-$dimensional subspace $K\subset E(\HH^k(G)),$ we set $\beta=1+\epsilon$ for $\epsilon\in(0,\frac12).$ By Lemma~\ref{innerprod}, $A_R$ is an inner product on $K$ for $R\geq R_0(K).$ By Lemma \ref{PNGL1}, there exist infinitely many $R\geq R_0(K)$ such that for any orthonormal basis $\{u_i\}_{i=1}^m$ of $K$ with respect to $A_{(1+\epsilon) R},$ we have
$$\sum_{i=1}^m A_R(u_i,u_i)\geq m(1+\epsilon)^{-(2k+2+\delta)}.$$
Lemma \ref{PNGL2} implies that $$\sum_{i=1}^m A_R(u_i,u_i)\leq C(D)\epsilon^{-1}.$$ Setting $\epsilon=\frac{1}{2k},$ and letting $\delta\rightarrow 0,$ we obtain
\begin{equation*}\label{PF11}m\leq 2Ck\left(1+\frac{1}{2k}\right)^{2k+2+\delta}\leq Ck.\end{equation*} Since the above estimate holds for any subspace $K$ of $E(\HH^k(G)),$
$$\dim \HH^k(G)=\dim E(\HH^k(G))\leq Ck.$$ This proves the result.

\end{proof}



\begin{definition} We say that $u:V\times (-\infty,0]\to\R$ is an ancient solution to the heat equation if $u(x,\cdot)\in C^1((-\infty,0])$ for any $x\in V$ and
$$\partial_t u(x,t)=\Delta u(x,t),\quad \forall x\in V, t\leq 0.$$
\end{definition}
Fix $x_0\in V.$ We denote by
$$\mathcal{P}^k(G):=\{\mathrm{ancient\ solution}\ u: |u(x,t)|\leq C(1+d(x,x_0)+\sqrt{|t|})^k, \forall x\in V,t\leq 0\}$$ the space of ancient solutions of polynomial growth with growth rate at most $k.$
Note that the above space doesn't depend on the choice of $x_0.$

Ancient solutions of polynomial growth were studied by many authors in Riemannian geometry, \cite{CalleMZ06,Callethesis,LinZhang17,ColdingM19}. The following result is a discrete analog of Colding-Minicozzi's theorem \cite{ColdingM19}.
\begin{theorem}[\cite{Huaancient19}] Let $G=(V,E)$ be a graph with bounded vertex degree, which has polynomial volume growth, i.e. for some $x\in V, C>0,a>0,$ $$|B_R(x)|\leq CR^a,\ \forall R\geq 1.$$ Then for any $k\in \N,$
$$\dim\mathcal{P}^{2k}(G)\leq (k+1)\dim \HH^{2k}(G).$$
\end{theorem}

This yields the corollary of Theorem~\ref{thm:main1}.
\co Let $G$ be an infinite penny graph with bounded facial degree.
Then for any $k\geq 1,$$$\dim\mathcal{P}^{k}(G)\leq C(D) k^2.$$
\cod


\bibliography{py2}
\bibliographystyle{alpha}

\bigskip
\bigskip

\bigskip

\end{document}